\documentclass[11pt]{amsart}

\RequirePackage{amssymb}
\RequirePackage{enumerate}
\RequirePackage{amsfonts}
\RequirePackage{amsmath}
\RequirePackage{mathrsfs}
\RequirePackage{dsfont}
\allowdisplaybreaks[1]

\newcounter{fig}

\newcommand{\mybic}{\author{Gianluca Cassese}
                     \address{Universit\`{a} Milano Bicocca}
                     \email{gianluca.cassese@unimib.it}
                     \curraddr{Department of Economics, Statistics and Management 
                                  Building U7, Room 2097, via Bicocca 
                                  degli Arcimboldi 8, 20126 Milano - Italy}}

\newcommand{\Pred}{\mathscr{P}}

\newcommand{\OX}{\widetilde\Omega_X}

\newcommand{\B}{\mathfrak{B}} 
 
\newcommand{\F}{\mathscr{F}}
\newcommand{\R}{\mathbb{R}} 
\newcommand{\N}{\mathbb{N}}

\newcommand{\Bor}{\mathscr{B}}

\newcommand{\cadlag}{\textit{c\`{a}dl\`{a}g} }
\newcommand{\caglad}{\textit{c\`{a}gl\`{a}d} }

\newcommand{\qtext}[1]{\quad\text{#1}\quad}

\newcommand{\abs}[1]{\vert #1\vert} 
\newcommand{\babs}[1]{\big\vert #1\big\vert} 
 
\newcommand{\Babs}[1]{\Big\vert #1\Big\vert}

\newcommand{\inner}[2]{\langle #1,#2\rangle} 
 
\newcommand{\inn}[1]{\inner{#1}{#1}}

\newcommand{\net}[3]{\langle #1_{#2}\rangle_{#2\in #3} } 
 
\newcommand{\nnet}[3]{\langle #1\rangle_{#2\in #3} } 
 
\newcommand{\seq}[2]{\net{#1}{#2}{\mathbb{N}}} 
\newcommand{\sseq}[2]{\nnet{#1}{#2}{\mathbb{N}}} 
\newcommand{\seqn}[1]{\seq{#1}{n}}

\newcommand{\set}[1]{\mathds{1}_{#1}}
\newcommand{\sset}[1]{\mathds{1}_{\{#1\}}}

\newcommand{\cl}[2][ ]{\overline{#2}^{\ #1}}

\newcommand{\Fun}[1]{\mathfrak{F}(#1)}

\newcommand{\emp}{\varnothing}

\newcommand{\tiref}[1]
	{(\textit{#1})}
\newcommand{\timply}[2]
	{(\textit{#1})$\Rightarrow$(\textit{#2})}

\newtheorem{theorem}{Theorem}
\theoremstyle{plain}

\newtheorem{corollary}{Corollary}

\newtheorem{definition}{Definition}
\newtheorem{example}{Example}

\newtheorem{lemma}{Lemma}

\newcommand {\La}{\mathscr L}
\newcommand {\Ta}{\mathscr T}
\newcommand {\Ha}{\mathscr H}

\begin{document} 
\title[Tanaka' Formula]
{A Generalisation of the formulas of \^Ito and Tanaka}
\mybic
\date \today 
\subjclass[2010]{
  Primary: 60G07, 60H05;
  Secondary: 60J55.
  } 
\keywords{
  Conglomerability,
  \^Ito's lemma, 
  Semimartingale,
  Stochastic calculus,
  Tanaka's formula.}

\begin{abstract} 
We prove a generalization of the formulas 
of \^Ito and Tanaka to Lipschitz continuous 
functions.
\end{abstract}

\maketitle

\section{Introduction.}
Since the early treatment given by Meyer in \cite{meyer} 
and later in greater detail by Dellacherie and Meyer 
\cite{dellacherie meyer}, the classical formulas 
of \^Ito and Tanaka are among the cornerstones 
of stochastic analysis. Roughly put, these state 
that semimartingales are a stable class with respect 
to twice continuously differentiable and to convex 
transformations respectively and provide a detailed 
description of the characteristics of the transformed 
semimartingale in terms of those of the original one.

In this paper we obtain an extension of these 
classical results to the case of Lipschitz continuous 
functions satisfying some restrictions concerning 
the second difference. Thus our results do not 
require differentiability, a circumstance that may
prove to be useful in applications. Nevertheless,
the starting point of our approach is an expansion
formula for Lipschitz functions which is quite similar
to Taylor expansion with integral remainder.

Technically speaking, the main tool employed 
in the following results is a general integral 
representation theorem developed in \cite{conglo}
and based on the notion of conglomerability.
The use of this representation is pervasive
in our analysis and brings with itself some
new features with respect to the traditional
approach. We find, e.g., quite convenient to
work with stochastic processes which have
a double time index due to the need, in our
extended setting, to work with incremental
ratios rather than derivatives.

In section 1 we prove our version of Taylor
expansion for Lipschitz functions. In the
following section 2 we define the notion of
predictable compensator for increasing
processes which need not be locally integrable.
In section 3 we prove the main results
concerning the representation of $X$-summable
functions. From these we deduce in the following
section 4 the extension of the formulas of 
\^Ito and of Tanaka to Lipschitz functions.

As for notation, we write the expected value 
of a measurable quantity $f$ (somehow 
unconventionally) as $P(f)$ and we use the 
symbols $\Fun A$ and $\B(A)$ to denote the 
class of real valued functions and of bounded 
functions on a set $A$, respectively.

All stochastic processes mentioned in the
following sections are defined on a fixed,
filtered probability space 
$\big(\Omega,\F,P;(\F_t:t\ge0)\big)$ satisfying
the usual assumptions of completeness and
right continuity. $\Ta$ will be the family of
all stopping times of the filtration which are
finite $P$ a.s. and $\Pred$ the predictable
$\sigma$ algebra.

\section{Summable functions on the real line}

We start by developing in this section an expansion
formula for Lipschitz continuous functions, quite
similar to Taylor expansion. The novel features 
of the approach presented are that do not require
differentiability and we take as our starting point
the somehow unconventional family $\Fun{\R^2}$ 
of real-valued functions on $\R^2$. Some of the
properties of real valued functions extend to
this new class in a rather natural way.

We starting defining incremental ratios inductively: 
if $F\in\Fun{\R^2}$ and $h_1,h_2,\ldots\in\R_{++}$ 
let
\begin{equation}
F^1_{h_1}(x)
=
\frac{F(x,x+h_1)}{h_1}
\qtext{and}
F^k_{h_1,\ldots,h_k}(x)
=
\frac{F^{k-1}_{h_1,\ldots,h_{k-1}}(x+h_k)
-
F^{k-1}_{h_1,\ldots,h_{k-1}}(x)}
{h_k}
\qquad
x\in\R.
\end{equation}
We also write $F^k_{h_1,\ldots,h_k}(x)$ as 
$F^k(x,h_1,\ldots,h_k)$ when regarding $F^k$ as 
a function of $\R\times\R_{++}^k$. 

The following definitions mimic those valid for 
$\Fun\R$. Denote by $\Pi_a^b$ the family of 
all finite collections $\pi=\{x_1\le\ldots\le x_N\}$ 
in $(a,b)$, directed by inclusion.

\begin{definition}
A function $F\in\Fun{\R^2}$ is said to be:
\begin{enumerate}[(i)]
\item
summable -- in symbols $F\in\mathscr I$ -- if 
$F(x,x)=0$ for all $x\in\R$ and the limit
\begin{equation}
\label{IF}
I(F;a,b)
	=
\lim_{\pi\in\Pi_a^b}\sum_{x_{n-1},x_n\in\pi}
F(x_{n-1},x_n)%
\footnote{
$I(F;-\infty,+\infty)$ is simplified into $I(F)$.
}
\end{equation}
exists and is finite for all $-\infty\le a<b\le+\infty$;
\item
of finite variation -- in symbols $F\in\mathscr V$ -- if 
$F\in\mathscr I$ and for all $-\infty\le a<b\le\infty$
\begin{equation}
\label{VF}
V(F;a,b)
	=
\limsup_{\pi\in\Pi_a^b}\sum_{x_{n-1},x_n\in\pi}\abs{F(x_{n-1},x_n)}
	<
\infty;
\end{equation}
\item
Lipschitz -- in symbols $F\in\mathscr L$ -- if 
$F\in\mathscr I$ and
\begin{equation}
\limsup_{h_1+\ldots+h_k\downarrow0}\ 
\sup_{x\in\R,0<t\le h_k}
\babs{F^k(x,h_1,\ldots,h_{k-1},t)}
<
\infty
\qquad
k\in\N.
\end{equation}
\end{enumerate}
\end{definition}

Of course $\La\subset\mathscr V\subset\mathscr I$,
\begin{equation}
I(F;a,b)=I(F;a,t)+I(F;t,b)
\qtext{and}
V(F;a,b)=V(F;a,t)+V(F;t,b)
\qquad
a\le t\le b
\end{equation}
when $F$ belongs to the corresponding class,
$\mathscr I$ or $\mathscr V$ respectively. Thus 
$F\in\mathscr V$ implies $V(F)\in\mathscr I$
while $F\in\La$ implies $V(F)\in\La$.

Our interest for $\Fun{\R^2}$ is illustrated by the 
following examples. First, any function $f\in\Fun\R$ 
generates a corresponding function 
$\widehat f\in\mathscr I$ defined as
\begin{equation}
\label{hat f}
\widehat f(x,y)
	=
f(y)-f(x)
\end{equation}
which is of finite variation or Lipschitz in the above 
defined sense if and only if $f$ is so in the usual 
sense. The notation $\widehat f$ will be used 
repeatedly throughout. In general the elements
of $\mathscr I$ considered will be of the form
\begin{equation}
\label{f circ g}
F(x,y)
=
g(x)\widehat f(x,y)
\qquad
f,g\in\Fun\R
\end{equation}
or linear combinations thereof. Note that indeed
\eqref{f circ g} defines an element of $\mathscr I$
whenever $f$ is of finite variation and the integral
$\int gdf$ exists in the Remann-Stiltijes sense.
A combination of $f,g\in\Fun\R$ of special
importance will be
\begin{equation}
\label{f*g}
(f*g)(x,y)=f(y)-f(x)-g(x)(y-x)
\end{equation}
which defines an element of $\mathscr I$ when
$g$ is Riemann integrable. It is also of interest
to note that $f*g\ge0$ if and only if $f$ is convex 
and $D_-f\le g\le D_+f$, with $D_-$ and $D_+$ 
the left and right derivatives of $f$.

\begin{theorem}
\label{th taylor}
Let $F\in\mathscr L$. Fix $b>a$ and $k\in\N$. 
There exist $\bar F^1,\ldots,\bar F^{k-1}\in\Fun\R$
and (finitely additive) probabilities $\lambda^k,\xi^k$ 
(depending on $a$ and $b$) on a ring of subsets 
of $[a,b]\times\R_{++}^k$
such that 
\begin{equation}
\label{taylor}
\begin{split}
I(F;a,b)
	&=
\sum_{j=1}^{k-1}\frac{(b-a)^j}{j!}\bar F^j(a)
+
\frac{(b-a)^k}{k!}
\int F^k(x,h_1,\ldots,h_k)d\lambda^k\\
	&=
-\sum_{j=1}^{k-1}\frac{(a-b)^j}{j!}\bar F^j(b)
-
\frac{(a-b)^k}{k!}
\int F^k(x,h_1,\ldots,h_k)d\xi^k
\end{split}•
\end{equation}

Moreover:
\begin{enumerate}[(a)]
\item
if $x\to F(x,x+t)$ is a Borel measurable of $x$ 
for each $t\in\R$ then $\bar F^j$ is measurable
for all $j\in\N$,
\item
$\lambda^k
\big(\R\times[t_1,\infty)\times\ldots\times[t_k,\infty)\big)
=
\xi^k
\big(\R\times[t_1,\infty)\times\ldots\times[t_k,\infty)\big)
=0$
for all $t_1,\ldots,t_k>0$ and
\item
if the limit
$D_+^jF(x)
	=
\lim_{h_1,\ldots,h_j\to0}F^j(x,h_1,\ldots,h_j)$
exists, then $\bar F^j(x)=D^j_+F(x)$.
\end{enumerate}
\end{theorem}

\begin{proof}
Regarding $I(F)$ as a linear functional 
on $\mathscr L$, we conclude from \eqref{IF} that
\begin{equation}
\label{con}
I(F)<0
\qtext{implies}
\inf_{(x,h)\in\R\times\R_{++}}F^1(x,h)<0
\qquad
F\in\La
\end{equation}
i.e. that $I$ is conglomerative with respect to
the linear map $T:\La\to\La$ associating each 
$F\in\La$ with $F^1$. Moreover, 
$\abs{T F}
	\le 
T V(F)$ 
and $V(F)\in\La$ (so that $T$ is directed on 
$\La$) and $T[\La]\subset\B(\R\times\R_{++})$. 
From \cite[Theorem 1]{conglo} we deduce the 
existence of a finitely additive, positive set 
function $\nu$ on $\R\times\R_{++}$ such 
that
\begin{equation}
\label{I(F) rep}
F^1\in L^1(\nu)
\qtext{and}
I(F)
	=
\int (TF)(x,h)\nu(dx,dh)
	=
\int F^1d\nu
\qquad
F\in\La.
\end{equation}
Denote by $\nu_a^b$ the restriction of $\nu$ to 
$[a,b]\times\R_{++}$. Consider 
$F(x,y)=f(y-x)g(x)(y-x)$ 
where $f$ is bounded and with right limit $f(0^+)$ 
at $0$ and $g$ is continuous and with compact 
support. Then,
\begin{equation}
\label{margin}
F\in\La
\qtext{and}
f(0^+)\int g(t)dt
	=
I(F)
	=
\int F^1d\nu
	=
\int f(h)g(x)d\nu.
\end{equation}
Thus the $x$-marginal of $\nu$ coincides, on 
$\Bor(\R)$, with Lebesgue measure while the 
normalization of the $h$-marginal of $\nu_a^b$ 
is the probability associated with the ultrafilter
$\{D\subset\R_{++}:0\in\cl D\}$ and, being
independent of $a$ and $b$, we shall denote 
it by $\mu$.

Formula \eqref{taylor} is then deduced from 
the integral representation \eqref{I(F) rep} in 
much the same way as Taylor expansion follows 
from Newton-Leibniz Theorem. Setting $x_0=b$ 
and using definition \eqref{hat f}, we can develop 
\eqref{I(F) rep} into
\begin{equation}
\label{expand}
\begin{split}
I(F;a,b)
	&=
\int F^1(a,h_1)d\nu_a^{x_0}
+
\int I\big(\hat F^1_{h_1};a,x_1\big)d\nu_a^{x_0}\\
	&=
\sum_{j=1}^{k-1}
\int F^j(a,h_1,\ldots,h_j)
(\nu_a^{x_{j-1}}\times\ldots\times\nu_a^{x_0})
(dx_j,dh_j,\ldots,dx_1,dh_1)\\
&\quad+
\int F^k(x,h_1,\ldots,h_k)
(\nu_a^{x_{k-1}}\times\ldots\times\nu_a^{x_0})
(dx_k,dh_k,\ldots,dx_1,dh_1)\\
	&=
\sum_{j=1}^{k-1}\frac{(b-a)^j}{j!}
\int F^j(a,h_1,\ldots,h_j)d\mu^j
+
\frac{(b-a)^k}{k!}
\int F^k(x,h_1,\ldots,h_k)d\lambda^k
\end{split}
\end{equation}
where $\lambda^k$ is the normalisation of 
$\nu_a^{x_{k-1}}\times\ldots\times\nu_a^{x_0}$ 
and $\mu^j$ denotes the $j$ fold product of $\mu$.
Set
\begin{equation}
\label{bar F^j}
\bar F^j(a)
=
\int F^j(a,h_1,\ldots,h_j)d\mu^j.
\end{equation}
Alternatively, we could have expanded $I(F;a,b)$ in
\eqref{expand} (with $x_0=a$ this time) as
\begin{equation}
\begin{split}
I(F;a,b)
&=
\int F^1(b,h)d\nu_a^b
-
\int I(\hat F^1_{h_1};x_1,b)d\nu_a^b\\
&=
\sum_{j=1}^{k-1}(-1)^{j+1}
\int F^j(b,h_1,\ldots,h_j)
(\nu_{x_{j-1}}^b\times\ldots\times\nu_{x_0}^b)\\
&\quad+
(-1)^{j+1}\int F^k(x_k,h_1,\ldots,h_k)
(\nu_{x_{k-1}}^b\times\ldots\times\nu_{x_0}^b)\\
&=
-\sum_{j=1}^{k-1}\frac{(a-b)^j}{j!}\bar F^j(b)
-
\frac{(a-b)^k}{k!}\int F^k(x,h_1,\ldots,h_k)d\xi^k
\end{split}
\end{equation}
where $\xi^k$ is the normalisation of 
$\nu_{x_{k-1}}^b\times\ldots\times\nu_{x_0}^b$.
If $x\to F(x,x+t)$ is Borel measurable for each 
$t\in\R$ then $a\to F^j(a,h_1,\ldots,h_j)$ is 
measurable for all $j\in\N$ and 
$h_1,\ldots,h_j\in\R_{++}$. 
But then the integral in \eqref{bar F^j} is just the 
limit of sums which are measurable as a function 
of $a$.
Claim \tiref{b} is a simple consequence of $\mu$ 
being an ultrafilter probability
while \tiref{c} follows from ultrafilter limits
coinciding with ordinary limits whenever the 
latter exist.
\end{proof}

When $F=\widehat f$ the quantity $\bar F^j$
will be called the approximate $j$-th derivative
of $f$.


Eventually, \eqref{taylor} extends from $\La$ to
$\mathscr I$. In this case the representation 
\eqref{I(F) rep} would take the form
\begin{equation}
I(F)
	=
\phi(F^1)+\int F^1d\nu
\end{equation}
with $\phi$ a positive linear functional such that 
$\inf_{x,h}F^1(x,h)>-\infty$ implies $\phi(F^1)\ge0$. 
The formula \eqref{taylor} would then become 
considerably more involved although, with $k=1$, 
we would still get a manageable expression, such 
as
\begin{equation}
I(F,a,b)
=
\phi_a^b(F^1)
+
(b-a)\bar F^1(a)
+
(b-a)\int I(\hat F_{h_1},a,x_1)d\nu_a^b.
\end{equation}

\section{The predictable compensator}
Before examining functions which are summable 
along a semimartingale we need to extend the
notion of predictable compensator, customarily
defined for processes of locally integrable variation 
only, to the wider class of processes of finite variation.

\begin{lemma}
\label{lemma compensator}
An adapted, right continuous, increasing process 
$A$ admits one and only one (up to evanescence) 
predictable, right continuous, increasing process 
$A^p$ satisfying either one of the following two 
equivalent properties:
\begin{enumerate}[(i)]
\item
for every predictable process $Y$,
\begin{equation}
\label{compensator process}
P\int\abs YdA<\infty
\qtext{implies}
P\int YdA
=
P\int YdA^p;
\end{equation}
\item
for some disjoint family $B_1,B_2,\ldots$ in 
$\Pred$ and $B_0=\bigcap_nB_n^c$ the 
following jointly holds: 
\begin{subequations}
\label{compensator set}
\begin{equation}
\label{compensator set1}
P\int\set{B\cap B_0}dA\in\{0,+\infty\}
\qtext{while} 
P\int\set{B\cap B_0}dA^p=0
\qquad
B\in\Pred
\end{equation}
\begin{equation}
\label{compensator set2}
P\int \set{B\cap B_n}dA
	=
P\int \set{B\cap B_n}dA^p<\infty
\qquad
B\in\Pred,\ n\in\N.
\end{equation}
\end{subequations}
\end{enumerate}
\end{lemma}

In the sequel $A^p$ will be referred to as the 
predictable compensator of $A$ and we shall
write $\inn X=[X,X]^p$. We can then consider 
the $\sigma$ finite, countably additive set 
function
\begin{equation}
\label{m_X}
m_X(B)
=
P\int\set B\ d\inn X
\qquad
B\in\Pred
\end{equation}
without requiring further assumptions than $X$
to be a semimartingale. We will refer to the set
$B_0$ in \tiref{ii} as a degenerate set for the 
process $A$ -- or rather for the measure induced 
by it. Notice that $P\int\set Bd[X,X]=m_X(B)$
if and only if $[X,X]$ is $\sigma$ integrable.

\begin{proof}
Starting with the uniqueness claim, we notice
that \tiref{ii} implies
\begin{align*}
P\int\set BdA^p
	=
P\int\set{B\hspace{0.05cm}\cap\hspace{0.05cm}
B_0^c}dA^p
	=
\lim_NP\int\set{B\hspace{0.05cm}\cap\hspace{0.05cm} 
\bigcup_{n=1}^NB_n}dA
=
P\int\set{B\hspace{0.05cm}\cap\hspace{0.05cm} 
B_0^c}dA
\qquad
B\in\Pred.
\end{align*}
If $\hat B_0,\hat B_1,\ldots,$ and $\hat A^p$ also 
satisfy \eqref{compensator set}, we deduce likewise 
$P\int\set Bd\hat A^p
	=
P\int\set{B\cap\hat B_0^c}dA$.
However, since both $B_0$ and $\hat B_0$ are 
degenerate sets for $A$, we obtain
$
P\int\set{B_0\cap \hat B_0^c}dA
	=
P\int\set{\hat B_0\cap B_0^c}dA
	=
0
$
so that
\begin{align*}
P\int\set BdA^p
	=
P\int\set{B\cap B_0^c\cap \hat B_0^c}dA
	=
P\int\set Bd\hat A^p
\qquad
B\in\Pred
\end{align*}
and $A^p$ and $\hat A^p$ coincide up to evanescence
\cite[proposition I.2.8]{js}. 

Write $L^1(A)$ for the space of predictable processes
$Z$ satisfying $P\int\abs ZdA<\infty$ and
\begin{equation}
\mathfrak Y
=
\Big\{Y\in L^1(A):
1\ge Y\ge0,
Y
\text{ is left continuous}\Big\}.
\end{equation}•
For each $Y\in\mathfrak Y$ one may regard the 
integral $P\int YdA$ as an element $\lambda_Y$ 
of $ca(\Pred)_+$. We claim that the set
$
\{\lambda_Y:
Y\in\mathfrak Y\}
$
is dominated. In fact, let the collection
$\{B_\alpha:\alpha\in\mathfrak A\}
\subset
\Pred$ 
be disjoint and such that 
$\sup_{Y\in\mathfrak Y}\lambda_Y(B_\alpha)>0$ 
for all $\alpha\in\mathfrak A$. This implies
\begin{align*}
0
	<
P\int\set{B_\alpha}dA
	=
\lim_n\lim_k
P\Big(\sset{A_n\le k}\int_0^n\set{B_\alpha}dA\Big)
	=
\lim_n\lim_k\gamma_{n,k}(B_\alpha)
\end{align*}
(the last equality being just the definition of 
$\gamma_{n,k}\in ca(\Pred)$) i.e. that 
$\gamma_{n,k}(B_\alpha)>0$ for all
$\alpha\in\mathfrak A$ and some $n,k\in\N$. 
However, each $\gamma_{n,k}$ vanishes on all 
but countably many elements of 
$\{B_\alpha:\alpha\in\mathfrak A\}$
and since $\{\gamma_{n,k}:n\in\N,k\in\N\}$ is 
countable then $\mathfrak A$ has to be countable
too. The claim then follows from \cite[theorem 3]{hs}. 
Applying Halmos and Savage Theorem we obtain
a countable collection $Y_1,Y_2,\ldots\in\mathfrak Y$ 
such that, letting 
$Y_0
	=
\sum_n2^{-n}\frac{Y_n}{1+P\int Y_ndA}$,
\begin{equation}
\lambda_Y
	\ll
\lambda_{Y_0}
\qquad
Y\in\mathfrak Y.
\end{equation}
$Y_0$ is left continuous since the uniform limit of
left continuous functions; moreover, 
$P\int Y_0dA\le1$ so that $Y\in\mathfrak Y$. 
Since the stochastic integral $\int Y_0dA$ 
generates an increasing, integrable process,
it admits a predictable compensator $A_0^p$. 
Define 
\begin{equation}
\label{Ap}
A^p
	=
\int\frac{\sset{Y_0>0}}{Y_0}dA_0^p.
\end{equation}
Since $Y_0$ is locally bounded, $A^p$ is then 
predictable, increasing and right continuous. If 
$Y\in\mathfrak Y$, then
\begin{align}
\label{comp}
P\int YdA
	=
P\int Y\sset{Y_0>0}dA
	=
P\int Y\frac{\sset{Y_0>0}}{Y_0}dA_0
	=
P\int Y\frac{\sset{Y_0>0}}{Y_0}dA_0^p
	=
P\int YdA^p
\end{align}
which proves \eqref{compensator process} for
the case $Y\in\mathfrak Y$ and thus for all
processes belonging to the generated vector
space which is dense in $L^1(A)$ in the 
corresponding metric.

Of course, $A^p$ is integrable along some 
localizing sequence $\seqn\tau$.  Let
\begin{equation*}
Z
	=
\sum_n
\frac{\set{]]\tau_{n-1},\tau_n]]}}
{1\vee P\big(A^p_{\tau_n}-A^p_{\tau_{n-1}}\big)}.
\end{equation*}
Then $Z\in L^1(A^p)$ and $0\le Z\le1$: write
$\lambda_Z^p$ for the corresponding set 
function. We notice that 
$\{\lambda_Y:Y\in L^1(A)\}$ 
is dominated by $\lambda_Z^p$. Again by 
Halmos and Savage we obtain that the sets 
of the form $\{a\le Y\le b\}$ with
$Y\in L^1(A)_+$ and $a,b>0$ themselves
generate a dominated family of measures
and admits then a dominating subfamily,
$\{a_n<Y_n< b_n\}$ with $n=1,2,\ldots$.
Letting $B_1,B_2,\ldots$ be the disjoint
sequence induced by these sets and
$B_0=\bigcap_nB_n^c$, we conclude that
$\sup_{Y\in L^1(A)}\lambda_Y(B_0)=0$
which rules out the case 
$0
	<
P\int \set{B\cap B_0}dA
	<
\infty$
for all $B\in\Pred$, thus proving \eqref{compensator set1}. 
Moreover,
\begin{align*}
P\int\set{B_n}dA
	\le
P\int\sset{a_n\le Y_n\le b_n}dA
	\le 
\frac{1}{a_n}P\int YdA
	<
\infty
\end{align*}
so that also \eqref{compensator set2} follows from
\eqref{compensator process}. We have thus shown
that \timply{i}{ii}. 

Conversely, if \tiref{ii} is true and $Y\in L^1(A)_+$ 
then necessarily $P\int\set{B_0}YdA=0$ so that
\begin{align*}
P\int YdA
	=
\lim_N\sum_{n=1}^N
P\int(Y\wedge N)\set{B_n}dA
	=
\lim_N\sum_{n=1}^NP\int(Y\wedge N)\set{B_n}dA^p
	=
P\int YdA^p
\end{align*}
that is \timply{ii}{i}. 
\end{proof}

Condition \tiref{ii} is a special version of the 
result proved by Luther \cite[theorem 1]{luther} 
that $[0,+\infty]$ valued measures split (not 
uniquely)  into the sum of a degenerate and 
of a semifinite measure. In our case we rather 
obtain $\sigma$ finiteness and uniqueness of $A^p$.
In passing we notice that $A^c$ clearly induces a
$\sigma$ finite measure so that the degenerate part
of $A$ is a pure jump process. $P\int\set{B_0}dA=0$ 
if and only if $A$ is $\sigma$ integrable (i.e. the
induced set function on $\Pred$ is $\sigma$ finite)
while the case in which $A$ is locally integrable
corresponds to $B_0=\emp$ and 
$B_n=\ ]]\tau_{n-1},\tau_n]]$. 
This shows that indeed the construction achieved
in Lemma \ref{lemma compensator} is an extension
of the classical notion of compensator.

Eventually we notice that $P\int \abs YdA^p<\infty$
need not imply $P\int \abs YdA<\infty$.

\section{Summable functions along a semimartingale}
The above notion of summability and its implications
hinge on the properties of the real line and, in order 
to obtain an extension to other domains it needs to 
be modified conveniently. In particular, the extension 
to semimartingales requires two main changes in the 
above approach: first the definition of summability has 
to take into account topology and measurability and, 
second, incremental ratios as defined above partly loose 
their importance as long as the underlying process is not 
of finite variation.

Write 
$\widetilde\Omega
=
\Omega\times\R_+\times\R_+$.

\begin{definition}
A function 
$F
	\in
\Fun{\widetilde\Omega}$ is totally $X$-summable,
in symbols $F\in\mathscr{I}_X^*$, if for all 
$\sigma,\tau\in\Ta$
(i)
$F_{\sigma,\tau}$ is $\F$ measurable,
(ii) 
$P\big(F_{\sigma,\tau}\ne0;X_\sigma=X_\tau\big)=0$
and
(iii)
the limit 
\begin{equation}
\label{IX}
I_X(F,\infty)
	=
\lim_k\sum_n
F\big(T_{n-1}^k,T_n^k\big)
\end{equation}
exists in probability and is independent of the 
intervening Riemann sequence $\sseq{T_n^k}k$. 
\end{definition}

Although, for simplicity, we will hide $\omega$ 
as an input to $F\in\Fun{\widetilde\Omega}$, if
$F\in\mathscr I_X$ and $b\in\Fun\Omega$ then
$Fb\in\mathscr I_X$ and 
$I_X(Fb,\sigma)
	=
I_X(F,\sigma)b$. 
Of course if $F\in\mathscr I_X$ and $\sigma\in\Ta$
then $F^\sigma\in\mathscr I_X$ and we write
$I_X(F,\sigma)=I_X(F^\sigma,\infty)$ and $I_X(F)$
for the stochastic process with value
$I_X(F,t)$ at $t$.

An element of $\mathscr I_X$ of special importance
is 
\begin{equation}
\label{F0}
F_0(x,y)
=
(y-x)^2.
\end{equation}
Obviously, $I_X(F_0,\infty)=[X,X]_\infty$. 
Heuristically, it is tempting to interpret 
$I_X(F)$ as a sort of generalized stochastic 
integral with respect to $X$. Returning 
to the preceding examples, when 
$F=\widehat{f\circ X}$ and $f\in\Fun\R$, 
as in \eqref{hat f}, then 
$I_X(F,\infty)=f(X_\infty)-f(X_0)$. 
If $f,g\in\Fun\R$, $g$ is measurable and 
$F=f*g$ then
$I_X(F,\infty)
	=
f(X_\infty)-f(X_0)-\int (g\circ X)_-\ dX$.

Set
$\OX
=
\{(\omega,t,u)\in\widetilde\Omega
:
X_t\ne X_u
\}$
and define the map $T_X:
\mathscr I_X
\to
\Fun{\OX}$
implicitly via
\begin{equation}
\label{TX}
T_X(F)(t,u)
	=
\frac
{F\big(t,u\big)}
{\big(X_u-X_t\big)^2}.
\end{equation}

In order to obtain that the map $T_X$ is directed, 
we consider the following two subspaces of 
$\mathscr I_X$:
\begin{subequations}
\begin{equation}
\label{Ha}
\Ha_X^*
	=
\left\{
F\in\mathscr I_X^*:
I_X(F,\infty)\in L^1(P)
\text{ and }
T_X(F)\in\B(\OX)
\right\}
\qtext{and}
\end{equation}
\begin{equation}
\label{La}
\La_X^*
	=
\{F\in\mathscr I_X^*:
I_X(F,\infty)\in L^1(P),
\abs F\le G
\text{ where }
G\in\mathscr I_X^*
\text{ and }
I_X(G,\infty)\in L^1(P)
\}.
\end{equation}
\end{subequations}
When $[X,X]_\infty\in L^1(P)$ the inclusion 
$F_0\in\Ha_X^*$ implies that $T_X$ is a 
directed map on either space.

\begin{example}
\label{ex lip}
Let $g\in\Fun\R$ be Lipschitz on any interval and $f$ 
one of its primitives. If $2K$ is the Lipschitz constant 
for $g$ on the interval $U$ then
\begin{align*}
\abs{f(y)-f(x)-g(x)(y-x)}
	&=
\Babs{\int_x^y[g(t)-g(x)]dt}
	\le
K(y-x)^2
\qquad
x,y\in U
\end{align*}
so that $\abs{T_X(f*g)}\le K$ when $X$ takes 
values in $U$.
\end{example}

We prove now a general representation theorem for
totally $X$-summable functions. 

\begin{theorem}
\label{th rep}
Let $[X,X]$ be locally integrable. There exists a 
unique linear map 
$T_X^\Pred:\Ha_X^*\to L^\infty(m_X)$
satisfying:
(i)
$T_X^\Pred(Hb)=T_X^\Pred(H)b$ when 
$b\in L^\infty(m_X)$,
(ii) 
$H\ge G$ up to a $P$ evanescent set implies 
$T_X^\Pred(H)\ge T_X^\Pred(G)$, $m_X$ a.s.
and 
(iii)
the following representation holds:
\begin{equation}
\label{rep}
P\big( I_X(H,\infty)\big)
	=
P\int T_X^\Pred(H)\ d\inn X
\qquad
H\in\Ha_X^*.
\end{equation}
\end{theorem}

\begin{proof}
First we notice that uniqueness follows easily.
If $\hat T_X^\Pred$ were another such operator,
then \eqref{rep} would imply that
$
0
	=
P\int[\hat T_X^\Pred(F)-T_X^\Pred(F)]\set Bd\inn X
$
for every $B\in\Pred$, a conclusion extending to 
all $B\in\Pred$ by countable additivity. Thus,
$m_X\big(T^\Pred_X(F)>\hat T^\Pred_X(F)\big)
	=
0$.

We claim that
\begin{equation}
\label{phi<0}
P\big(I_X(H,\infty)\big)<0
	\qtext{implies}
\sup_{P(N)=0}\ 
\inf_{\substack{0\le t<u\\\omega\in N^c}}
T_X\big(H\big)(t,u)
	<
0.
\end{equation}
In fact, assume that
$
P\big(I_X(H,\infty)<-2\eta\big)
	>
\delta
$
for some $\eta,\delta>0$.
According to \eqref{IX}, along any Riemann 
sequence $\sseq{T_n^k}{k}$ and for $k$ 
sufficiently large, the inequality
\begin{equation}
\label{ineq}
\begin{split}
-\eta
	&>
\sum_n
H(T_{n-1}^k,T_n^k)\set{N^c}
\\
	&=
\sum_n\big(X_{T_n^k}-X_{T_{n-1}^k}\big)^2
\frac{H(T_{n-1}^k,T_n^k)}
{\big(X_{T_n^k}-X_{T_{n-1}^k}\big)^2}\set{N^c}
\\
	&\ge
\sum_n\big(X_{T_n^k}-X_{T_{n-1}^k}\big)^2
\inf_{\substack{0\le t<u\\\omega\in N^c}}
T_X(H)(t,u)
\end{split}
\end{equation}
obtains with probability at least as large as $\delta$,
for all $P$ null sets $N$. We deduce the claim 
from the fact that $P$ null sets form a collection
closed with respect to countable unions.

As in the proof of Theorem \ref{th taylor}, there exists 
then a positive, additive set function $\nu$ defined on 
some ring of subsets of 
$\OX$, 
satisfying $\nu(A)=0$ when the projection of 
$A$ on $\Omega$ is $P$ null,
\begin{equation}
T_X(H)\in L^1(\nu)
\qtext{and}
P\big(I_X(H,\infty)\big)
	=
\int T_X(H)d\nu
\qquad
H\in\Ha_X.
\end{equation}
For example, let $H\in\Ha_X^*$ be of the form
\begin{equation}
H(u,v)
	=
(X_v-X_u)^2Y_v
\end{equation}
with $Y$ a bounded, predictable and either \caglad
or \cadlag process. Then, $T_X(H)(t,u)=Y_u$ while, 
from standard theorems on stochastic integrals,
\begin{equation}
I_X(H,\infty)
	=
\int Y d[X,X].
\end{equation}
As in the proof of Theorem \ref{th taylor}, we 
conclude that the $\Pred$-marginal of $\nu$ 
coincides with $m_X$. More generally, fix 
$g\in L^1(\nu)_+$ and define
\begin{equation}
\lambda_g(B)
	=
\int g\set Bd\nu
\qquad
B\in\Pred.
\end{equation}
Choosing $\varepsilon>0$ arbitrary and $k$ large 
enough so that 
$\int (g\wedge k)d\nu
	\ge
\int gd\nu-\varepsilon$,
we conclude
\begin{align*}
\lambda_g(B)
	=
\int g\set Bd\nu
	\le
\varepsilon+km_X(B)
\end{align*}
i.e. $\lambda_g\ll m_X$. Since $m_X$ is $\sigma$
finite we can define $(g)^\Pred\in L^1(m_X)$ to 
be the corresponding Radon Nikodym derivative. 
Properties \tiref{i} -- \tiref{iii} above are clear. 
\end{proof}

The lack of exact derivatives forces to work with
doubly indexed stochastic processes, such as 
$T_X(F)$, a mathematical object of considerable
complexity. Nevertheless this unusual aspect
is easily approached via  integral representation 
adopted. Its main drawback is that we have 
limited information on the operator
$T^\Pred_X(F)$, despite some superficial 
resemblance with the notion of predictable 
projection. 

The restriction that $T_X(F)$ be locally bounded, 
embodied in the definition of the class $\Ha_X^*$, 
is unduly restrictive. We may partly circumvent it
by looking at the space $\La_X^*$.

\begin{theorem}
\label{th La*}
Let $[X,X]$ be locally integrable. There exist 
a linear map $T_X^\Pred:\La_X^*\to L^1(m_X)$ 
satisfying properties (i) -- (ii) of Theorem \ref{th rep} 
and
a positive linear functional $\phi$ such that
\begin{equation}
\label{rep La*}
P\big(I_X(F,\infty)\big)
=
\phi\big(T_X(F)\big)
+
P\int T_X^\Pred(F)d\inn X
\qquad
F\in\La_X^*
\end{equation}
where $\phi\big(T_X(F)\big)\ge0$ whenever 
$T_X(F)$ is uniformly lower bounded. 
\end{theorem}

\begin{proof}
The proof is quite similar to that of Theorem 
\ref{th rep}. In particular we can retain from 
that result the proof of claim \eqref{phi<0} 
and deduce from \cite[theorem 1]{conglo} 
the representation
\begin{equation}
P\big(I_X(F,\infty)\big)
	=
\phi\big(T_X(F)\big)
+
\int T_X(F)d\nu
\qquad
F\in\La_X^*,
\end{equation}
with $\phi$ a positive linear functional vanishing 
on $\B(\widetilde\Omega)$ and $\nu$ a positive, 
additive set function on some ring of subsets of 
$\widetilde\Omega$ which vanishes on $P$ 
evanescent sets and is such that 
$T_X(F)\in L^1(\nu)$. 
Restricting attention to $\Ha_X^*\subset\La_X^*$ 
we conclude that $\nu$ possesses exactly the same 
properties as the corresponding set function in 
Theorem \ref{th rep}. We then deduce as above that
\begin{equation}
\int T_X(F)d\nu
=
P\int T_X^\Pred(F)d\inn X.
\end{equation}
\end{proof}

We obtain considerable more information for
$X$-summable functions defined as follows:

\begin{definition}
A function $F\in\Fun{\widetilde\Omega}$ is 
$X$-summable, in symbols $X\in\mathscr I_X$, 
if for all $\sigma,\tau\in\Ta$
(i)
$F_{\sigma,\tau}$ 
is $\F_{\sigma\vee\tau}$ measurable and
(ii)
$F^\sigma\in\mathscr I_X^*$ 
where
\begin{equation}
\label{Fs}
F^\sigma(u,v)
=
F(u\wedge\sigma,v\wedge\sigma)
\qquad
u,v\in\R_+.
\end{equation}
\end{definition}
If $F\in\mathscr I_X$ and $\sigma\in\Ta$ then 
we write 
$I_X(F,\sigma)
=
I_X(F^\sigma,\infty)
$
and denote by $I(F)$ the corresponding 
stochastic process which is adapted by
definition. 

We may define likewise the classes
\begin{subequations}
\begin{equation}
\Ha_X
=
\big\{F\in\mathscr I_X:
F^{\sigma}\in\Ha_X^*
\text{ for all }
\sigma\in\Ta\big\}
\qtext{and}
\end{equation}
\begin{equation}
\La_X
=
\big\{F\in\mathscr I_X:
F^{\sigma}\in\La_X^*
\text{ for all }
\sigma\in\Ta\big\}.
\end{equation}•
\end{subequations}•

\begin{theorem}
\label{th Ha}
Let $F\in\Ha_X$. Then, $I_X(F)$ is a process of 
locally integrable variation with compensator
$
\int T_X^\Pred(F)d\inn X
$. 
Moreover,
\begin{equation}
\label{representation}
I_X(F,t)
	=
\int_0^t T_X^\Pred(F)d\inn{X^c}
+
\sum_{s\le t}\Delta I_X(F,s).
\end{equation}
\end{theorem}

\begin{proof}
Write
$G(t,u)
=
F(t,u)\set{]]0,\sigma]]}(u)$.
Then $G\in\mathscr I_X$ and 
$I_X(F^\sigma,\infty)=I_X(G,\infty)$ so that 
\begin{align*}
P\big(I_X(F,\sigma)\big)
	=
P\int T^\Pred_X(G)d\inn X
	=
P\int T^\Pred_X(F)_u\set{]]0,\sigma]]}(u)\inn X_u
	=
P\int_0^\sigma T_X(F)d\inn X.
\end{align*}

Consider first the case in which $X$ is locally 
bounded and thus $[X,X]$ locally integrable. 
Fix $\sigma,\tau\in\Ta$ and let $b\in\B(\F)$ 
be such that 
$\abs{I_X(F,\tau)-I_X(F,\sigma)}
	=
I_X(Fb,\tau)-I_X(Fb,\sigma)
$.
Then,
\begin{align*}
P\big(\babs{I_X(F,\tau)-I_X(F,\sigma)}\big)
	=
P\int_\sigma^\tau T_X^\Pred(Fb)d\inn X
	\le
P\int_\sigma^\tau \babs{T_X^\Pred(F)}d\inn X.
\end{align*}
This proves that indeed $I_X(F)$ is of locally 
integrable variation and \cadlag in mean. By 
right continuity and completeness of the 
filtration, $I_X(F)$ admits then a \cadlag 
modification. From classical results, 
\cite[VI.13]{dellacherie meyer}, we further 
deduce that 
$M_t
=
I_X(F,t)-\int_0^t T_X^\Pred(F)d\inn{X}$
is a pure jump martingale so that
\begin{equation}
\sum_{s\le t}\Delta I_X(F,s)
	=
M_t
 +
\sum_{s\le t}T_X^\Pred(F)_s\Delta X_s^2
\qquad
t\in\R_+,
\end{equation}
from which \eqref{representation} readily follows. This 
sets the basis for an induction proof of the claim 
for $X$ general. As for the induction step, suppose 
now that $Y$ is a semimartingale satisfying 
\eqref{representation}, that $\tau\in\Ta$ and that
\begin{equation*}
Z_t
=
Y_t + D_\tau\set{[[\tau,+\infty[[}
\end{equation*}
with $D_\tau$ an $\F_\tau$ measurable, finite valued
random variable. $Z$ is then a semimartingale.
Fix $F\in\mathscr I_Z$. Upon computing $I_Z(F,t)$ along
a Riemann sequence which contains $\tau\wedge t$, we
easily obtain
\begin{equation*}
I_Z(F,t)
=
I_Z(F,t\wedge\tau)+I_Z(F\set{]]t\wedge\tau,+\infty[[},t)
=
I_Z(F,t\wedge\tau)+I_Y(F\set{]]t\wedge\tau,+\infty[[},t).
\end{equation*}
The first equality implicitly proves that 
$F\set{]]t\wedge\tau,+\infty[[}\in\mathscr I_Z$; 
the second one follows from the fact that 
$F\set{]]t\wedge\tau,+\infty[[}\in\mathscr I_Z$
is equivalent to
$F\set{]]t\wedge\tau,+\infty[[}\in\mathscr I_Y$. 
Clearly, $Y_t=Z_t$ on $]]0,\tau[[$ while 
$Y_t-Y_\tau=Z_t-Z_\tau$ on $]]\tau,+\infty[[$ so that 
on either stochastic interval $T_Y(\cdot)=T_Z(\cdot)$
as well as $m_Y=m_Z$ and therefore 
$T_Y^\Pred(\cdot)=T_Z^\Pred(\cdot)$. 

Thus 
\begin{align*}
I_Z(F,t)
	&=
I_Z(F,\tau\wedge t-)
+
\Delta I_Z(F,\tau\wedge t)
+
I_Y\big(F\set{]]\tau\wedge t,+\infty[[},t\big)\\
	&=
\int_0^{\tau\wedge t} T_Z^\Pred(F)d\inn{Z^c}
+
\sum_{s<\tau\wedge t}\Delta I_Z(F,s)
+
\Delta I_Z(F,\tau\wedge t)\\
&\quad+
\int_{\tau\wedge t}^t T_Y^\Pred(F)d\inn{Y^c}
+
\sum_{s\le t}\Delta I_Y(F\set{]]\tau\wedge t,+\infty[[},s)\\
	&=
\int_0^t T_Z^\Pred(F)d\inn{Z^c}
+
\sum_{s\le t}\Delta I_Z(F,s).
\end{align*}
This shows that if $Y$ meets \eqref{representation} 
then so does $Z$.

Thus if $X$ is a general semimartingale, $a>0$ 
and $T>0$ we may define 
\begin{align}
\label{jumps}
X^0_t
	=
X_t
-
\sum_{s\le t}\Delta X_s\sset{\abs{\Delta X_s}>a}
\qtext{and}
X^n
=
X^{n-1}
+
\Delta X_{T_n}\set{[[T_n,\infty[[}
\qquad
n=1,\ldots,N
\end{align}
where $T_1,\ldots,T_N$ are the random times  
exhausting the jumps of $X$ larger than $a$ on 
$[0,T]$.

Thus, $X^0$ satisfies the representation 
\eqref{representation} by the the first part of 
this proof while each $X^n$ is obtained from 
$X^{n-1}$ by adding a jump and thus meets 
\eqref{representation} by the induction step.
\end{proof}

The following result shows that it is still possible 
to obtain useful deductions even if $F\in\La_X$.

\begin{theorem}
\label{th La}
Let $F\in\La_X$ be such that $I_X(F)$ is \cadlag 
and that $T_X(F^\sigma)$ is lower bounded 
for each $\sigma\in\Ta$. Then, $I_X(F)$ is a 
submartingale of locally integrable variation 
representable as
\begin{equation}
\label{tanaka}
I_X(F,t)
	=
A^c_t
+
\int_0^t T_X^\Pred(F)d\inn{X^c}
+
\sum_{s\le t}\Delta I_X(F,s)
\qquad
t\in\R_+
\end{equation}
with $A^c$ a continuous, increasing process.
\end{theorem}

\begin{proof}
The proof is essentially the same as that of 
Theorem \ref{th Ha}, so we only 
sketch the salient points. Assume that $[X,X]$ 
is locally integrable. Then, from Theorem 
\ref{th La*}
\begin{align*}
P\big(I_X(F,\sigma)\big)
=
\phi\big(T_X(F)\set{]]0,\sigma]]}\big)
+
P\int_0^\sigma T_X^\Pred(F)d\inn X
\ge
P\int_0^\sigma T_X^\Pred(F)d\inn X.
\end{align*}
The proof that $I_X(F)$ is of locally integrable 
variation follows from the inequality
\begin{align*}
P\Big(\babs{I_X(F,\tau)-I_X(F,\sigma)}\Big)
\le
\phi\big(\abs{T_X(F)}\set{]]\sigma,\tau]]}\big)
+
\int\abs{T_X(F)}\set{]]\sigma,\tau]]}d\nu.
\end{align*}
Then, 
$Y=I_X(F)-\int T_X^\Pred(F)d\inn{X}$
is a process of locally integrable variation and,
since 
\begin{align*}
P(Y_\tau-Y_\sigma)
=
\phi\big(T_X(F)\set{]]\sigma,\tau]]}\big)
\ge
0
\qquad
\sigma,\tau\in\Ta,\
\sigma\le\tau,
\end{align*}
a submartingale too. We write it as $Y=M+A$
with $M$, its martingale part, a compensated 
sum of jumps and $A$ an increasing, predictable 
process. Looking at its jumps we find that
\begin{align*}
\sum_{s\le t}\big\{\Delta I_X(F,s)
-
T_X^\Pred(F)_s\Delta X_s^2\big\}
=
M+\sum_{s\le t}\Delta A_s
\end{align*}
which proves the claim.
\end{proof}

It is clear that both Theorems remain true when 
$F\in\Ha_X$ or $F\in\La_X$ locally.

\section{Applications}

In this section we specialize our preceding 
results, proving extensions of \^Ito's and 
Tanaka's formulas from the class of 
$\mathscr C^2$ and convex functions
respectively to the class of locally Lipschitz 
continuous functions.

Recall \eqref{taylor}. The assumptions of the 
following Corollary are clearly met by 
$\mathscr C^2$ functions.

\begin{corollary}
\label{cor Ito}
Let $f\in\Fun\R$ be Lipschitz on intervals and 
assume that
\begin{equation}
\label{bounded}
\sup_{x\in K}\limsup_{h_1,h_2}
\babs{\widehat f^{\ 2}(x,h_1,h_2)}
	<
\infty
\qquad
K\subset\R\text{ compact}.
\end{equation}
Then $f(X)$ is a semimartingale and, denoting 
by $g$ the approximate derivative of $f$,
\begin{equation}
\label{Ito}
\begin{split}
f(X_t)
	&=
f(X_0)
+
\int_0^t(g\circ X)_-dX
+
\int_0^t T_X^\Pred\big(f*g\big)d\inn{X^c}\\
	&\quad+
\sum_{s\le t}\big\{\Delta(f\circ X)_s-(g\circ X)_{s-}\Delta X_s\big\}.
\end{split}
\end{equation}
\end{corollary}

\begin{proof}
By \eqref{taylor} we have 
$(f*g)(u,v)
=
\frac{(v-u)^2}{2}\int_u^v\widehat f^{\ 2}(x,h_1,h_2)
d\lambda^2$. 
Given that $\lambda^2$ only charges sets of 
the form $\R\times(0,1/n]^2$ for all $n\in\N$, 
we conclude from \eqref{bounded} that $T_X(f*g)$ 
is bounded on bounded intervals. If $X$ is locally 
bounded then, by localization, we can assume 
that $[X,X]$ and $\int (g\circ X)_-dX$ are integrable 
and that $f*g\in\Ha_X$. We then deduce \eqref{Ito} 
from \eqref{representation}. The extension to the 
case in which $X$ is a general semimartingale is 
obtained as in Theorem \ref{th rep}.
\end{proof}

Corollary \ref{cor Ito} provides an exact expansion
of $f(X)$ even if $f$ is not differentiable. An 
application of Theorem \ref{th La} delivers the 
following generalization of Tanaka's formula.
The assumptions of the following Corollary are
clearly met by convex functions.
\begin{corollary}
\label{cor Tanaka}
Let $f\in\Fun\R$ be Lipschitz on intervals and 
assume that
\begin{equation}
\label{pos}
\inf_{x\in K}\liminf_{h_1,h_2}
\widehat f^{\ 2}(x,h_1,h_2)
	>
-\infty
\qquad
K\subset\R\ compact.
\end{equation}
Then $f(X)$ is a semimartingale and, denoting the
approximate derivative by $g$,
\begin{equation}
\label{tanaka}
\begin{split}
f(X_t)
	&=
f(X_0)
+
A_t^c
+
\int_0^t(g\circ X)_-dX
+
\int_0^tT_X^\Pred(f*g)d\inn{X^c}\\
&\quad+
\sum_{s\le t}
\big\{
\Delta(f\circ X)_t-(g\circ X)_{s-}\Delta X_s
\big\}
\end{split}•
\end{equation}
with $A^c$ a continuous, increasing process.
\end{corollary}

\begin{proof}
By \eqref{taylor} and \eqref{pos} we conclude 
that for each interval there exists a constant 
$k>0$ such that $\abs{f*g}\le f*g+kF_0$. Thus,
if $X$ is locally bounded then $f*g\in\Ha_X$ 
locally and $T_X(f*g)$ is locally lower bounded.
Then \eqref{tanaka} is a consequence of Theorem
\ref{th La}.
\end{proof}

\end{document}